\documentclass[11pt]{article}

\usepackage{fullpage}
\usepackage{amsmath,amsthm,dsfont,amsfonts,amssymb,fancyhdr,mathtools}
\usepackage{graphicx}
\usepackage{color}
\usepackage{cite}
\usepackage[numbers]{natbib}
\usepackage{cases}
\usepackage{enumerate}
\usepackage[colorlinks=true, allcolors=blue]{hyperref}
\usepackage{verbatim}
\usepackage{tikz}
\usepackage{authblk}

\theoremstyle{definition}
\newtheorem{definition}{Definition}[section]
\newtheorem{remark}[definition]{Remark}
\newtheorem{example}[definition]{Example}

\newtheorem{notation}[definition]{Notation}

\theoremstyle{plain}
\newtheorem{theorem}[definition]{Theorem}
\newtheorem{lemma}[definition]{Lemma}
\newtheorem{corollary}[definition]{Corollary}
\newtheorem{proposition}[definition]{Proposition}

\begin{document}
\title{Towards a classification of $1$-homogeneous distance-regular graphs with positive intersection number $a_1$}

\author[a, b]{Jack H. Koolen}
\author[a]{Mamoon Abdullah}
\author[c]{Brhane Gebremichel}
\author[d]{Jae-Ho Lee\thanks{Corresponding author}}

\affil[a]{\small{School of Mathematical Sciences,
University of Science and Technology of China,
Hefei, Anhui, 230026, PR China.}}
\affil[b]{\small{CAS Wu Wen-Tsun Key Laboratory of Mathematics,
University of Science and Technology of China,
Hefei, Anhui, 230026, PR China.}}
\affil[c]{\small{Department of Mathematics, Adigrat University, Adigrat,
Tigray, 7040, Ethiopia.}}
\affil[d]{\small{Department of Mathematics and Statistics, University of North Florida, Jacksonville, FL 32224, U.S.A}}

%\author{Draft}

\date{}
\maketitle
\newcommand\blfootnote[1]{%
\begingroup
\renewcommand\thefootnote{}\footnote{#1}%
\addtocounter{footnote}{-1}%
\endgroup}
\blfootnote{E-mail addresses: {\tt koolen@ustc.edu.cn}~(J.H.~Koolen),
~{\tt mamoonabdullah@hotmail.com}~(M.~Abdullah),\\
~{\tt brhaneg220@gmail.com}~(B.~Gebremichel),
~{\tt jaeho.lee@unf.edu}~(J.-H.~Lee)}

\begin{abstract}
\noindent
Let $\Gamma$ be a graph with diameter at least two. 
Then $\Gamma$ is said to be $1$-homogeneous (in the sense of Nomura) whenever for every pair of adjacent vertices $x$ and $y$ in $\Gamma$, the distance partition of the vertex set of $\Gamma$ with respect to both $x$ and $y$ is equitable, and the parameters corresponding to equitable partitions are independent of the choice of $x$ and $y$.
Assume that $\Gamma$ is $1$-homogeneous distance-regular with intersection number $a_1>0$ and diameter $D\geqslant 5$.
Define $b=b_1/(\theta_1+1)$, where $b_1$ is the intersection number and $\theta_1$ is the second largest eigenvalue of $\Gamma$.
We show that if intersection number $c_2$ is at least $2$, then $b\geqslant 1$ and one of the following (i)--(vi) holds: 
(i) $\Gamma$ is a regular near $2D$-gon, 
(ii) $\Gamma$ is a Johnson graph $J(2D,D)$, 
(iii) $\Gamma$ is a halved $\ell$-cube with $\ell \in \{2D,2D+1\}$, 
(iv) $\Gamma$ is a folded Johnson graph $\bar{J}(4D,2D)$, 
(v) $\Gamma$ is a folded halved $4D$-cube, 
(vi) the valency of $\Gamma$ is bounded by a function of $b$.
Using this result, we characterize $1$-homogeneous graphs with classical parameters and $a_1>0$, as well as tight distance-regular graphs.


\end{abstract}

%%%%%%%%%%%%%%%%%%%%%%%%%%%%%%%%%%%%%%%%%%%%%
%%%%%%%%%%%%%%%%%%%%%%%%%%%%%%%%%%%%%%%%%%%%%
%%%%%%%%%%%%%%%%%%%%%%%%%%%%%%%%%%%%%%%%%%%%%
\section{Introduction}
In this paper, we study distance-regular graphs that have the $1$-homogeneous property in the sense of Nomura \cite{Nomura1994}.
To motivate our results, we recall some preliminaries and background on $1$-homogeneous distance-regular graphs. 
For more details, refer to \cite{BCN, DaKoTa, Nomura1994}.

\smallskip
Throughout this paper, let $\Gamma$ denote a finite, undirected, connected, and simple graph. 
Let $V(\Gamma)$ denote the vertex set of $\Gamma$.
For two vertices $x,y \in V(\Gamma)$, the \emph{distance} $d(x,y)$ is the length of a shortest path from $x$ to $y$ in $\Gamma$.
The \emph{diameter} of $\Gamma$ is the maximum value of $d(x,y)$ for all pairs of $x,y \in V(\Gamma)$.
Let $D$ denote the diameter of $\Gamma$.
For an integer $0\leqslant i \leqslant D$ and a vertex $x \in V(\Gamma)$, let $\Gamma_i(x)$ denote the set of vertices in $\Gamma$ at distance $i$ from $x$.
Abbreviate $\Gamma(x)=\Gamma_1(x)$.
The subgraph of $\Gamma$ induced on the set $\Gamma(x)$ is called the \emph{local graph} of $\Gamma$ at $x$.
The graph $\Gamma$ is called \emph{locally} $\mathcal{P}$ whenever every local graph of $\Gamma$ has the property $\mathcal{P}$ (or belongs to the family $\mathcal{P}$).
For example, we might say that a graph is locally connected or locally a strongly regular graph.
For a pair of vertices $x,y \in V(\Gamma)$ with $d(x,y)=2$, the subgraph of $\Gamma$ induced on the set $\Gamma(x)\cap \Gamma(y)$ is called the \emph{$\mu(x,y)$-graph} of $\Gamma$. 
If this graph does not depend on the choice of $x$ and $y$ (up to isomorphism), then we simply call it the $\mu$-graph of $\Gamma$.
For an integer $k\geqslant 0$, we say that $\Gamma$ is \emph{regular with valency} $k$ (or \emph{$k$-regular}) if $|\Gamma(x)|=k$ for every $x\in V(\Gamma)$. 
For an integer $0\leqslant i \leqslant D$ and for a pair $x, y \in V(\Gamma)$ with $d(x,y)=i$ we define 
\begin{equation}\label{sets:CAB}
	C_i(x,y) := \Gamma_{i-1}(x) \cap \Gamma(y), \quad 
	A_i(x,y) := \Gamma_{i}(x) \cap \Gamma(y), \quad 
	B_i(x,y) :=\Gamma_{i+1}(x) \cap \Gamma(y),
\end{equation}
where $C_{0}(x,y):= \varnothing$ and $B_{D}(x,y):=\varnothing$.
Observe that $\Gamma(y)$ is the disjoint union of the vertex sets $C_i(x,y)$, $A_i(x,y)$, $B_i(x,y)$.
We say $\Gamma$ is \emph{distance-regular} whenever the cardinalities 
\begin{equation}\label{int numbers}
	c_i =|C_i(x,y)|, \qquad a_i = |A_i(x,y)|, \qquad b_i=|B_i(x,y)| \qquad (0\leqslant i \leqslant D)
\end{equation}
are constants and do not depend on the choice of $x$ and $y$.
Note that $c_0=a_0=b_D=0$, and $c_1=1$. 
Additionally, $\Gamma$ is regular with valency $k=b_0$, and $a_i+b_i +c_i = k$ $(0\leqslant i \leqslant D)$.
The numbers $a_i, b_i, c_i$ in \eqref{int numbers} are called the \emph{intersection numbers} of $\Gamma$, and the array $\{b_0, b_1,\ldots , b_{D-1}; c_1, c_2,\ldots , c_D\}$ is called the {\em intersection array} of  $\Gamma$.
We note that a distance-regular graph with diameter $D$ has exactly $D+1$ distinct eigenvalues \cite[Proposition 2.6]{DaKoTa}.

\begin{notation}
Unless otherwise specified, whenever we denote $\Gamma$ as a distance-regular graph, we use the following notation: $\Gamma$ has diameter $D$, valency $k$, and distinct eigenvalues $\theta_0 > \theta_1 > \ldots > \theta_D$. 
Moreover, the intersection numbers of $\Gamma$ are denoted by $\{c_i\}^D_{i=1}$, $\{a_i\}^D_{i=0}$, $\{b_i\}^{D-1}_{i=0}$, as shown in \eqref{int numbers}.
\end{notation}

\smallskip
Next, we recall the notion of the $i$-homogeneous property as introduced by Nomura \cite{Nomura1994}.
Let $\Gamma$ be a connected graph.
A partition $\pi = \{ C_1, C_2, \ldots, C_p \}$ of $V(\Gamma)$ is called \emph{equitable} whenever, for all $1 \leqslant i, j \leqslant p$, the number of neighbors of a vertex $x\in C_i$ in the set $C_j$ is independent of the choice of $x$.
In other words, for each pair of subsets $C_i$ and $C_j$ in $\pi$, the number $c_{ij} := | \Gamma(x) \cap C_j |$ is constant for all $x\in C_i$.  
These numbers $\{c_{ij}\}_{1\leqslant i,j \leqslant p}$ are called the \emph{parameters} of $\pi$.
We say $\Gamma$ has the \emph{$i$-homogeneous} property whenever, for every pair of vertices $x$ and $y$ at distance $i$, the partition of $V(\Gamma)$ according to the path-length distance to both $x$ and $y$ is equitable, and the parameters corresponding to equitable partitions are independent of the choice of $x$ and $y$; see Section \ref{sec:1-homoDRGs}. 
Graphs with the $i$-homogeneous property are simply said to be $i$-homogeneous.
Note that $\Gamma$ is $0$-homogeneous if and only if it is a distance-regular graph.
Moreover, if $\Gamma$ is $1$-homogeneous, then it is a distance-regular graph and also locally strongly regular.

\smallskip
In this paper, we focus on $1$-homogeneous distance-regular graphs.
Examples of such graphs include the Johnson graphs $J(2D, D)$, the bipartite distance-regular graphs, and the regular near $2D$-gons. 
We have some comments about the history of $1$-homogeneous distance-regular graphs.
Over the years, the $1$-homogeneous property has received considerable attention and has been used in the study of distance-regular graphs across various contexts, including tight distance-regular graphs \cite{JKT2000, KLLLLT2023+}, distance-regular graphs which support a spin model \cite{CurtinNomura2004}, the Terwilliger algebras \cite{CurtinNomura2005}, and $Q$-polynomial distance-regular graphs \cite{Miklavic2004}.
Juri\v{s}i\'{c} and Koolen \cite{JK2000-1, JK2003, JK2011} explored $1$-homogeneous graphs whose $\mu$-graphs are a complete multipartite graph $K_{t\times n}$, $n\geqslant 1$ (i.e., the complement of $t$ copies of the complete graph $K_n$).
In \cite{JK2000-1}, they introduced the CAB property to study the local structures of distance-regular graphs and used this property to characterize $1$-homogeneous graphs with $a_1 > 0$. 
Also, they classified $1$-homogeneous graphs with $c_2\geqslant 2$ whose $\mu$-graphs are $K_{t\times 1}$, i.e., $1$-homogeneous Terwilliger graphs.
In \cite{JK2003}, they classified $1$-homogeneous graphs when $n=2$, i.e., when the $\mu$-graphs are $K_{t\times 2}$ (Cocktail Party graphs).
In their subsequent study \cite{JK2007}, they extended this work to distance-regular graphs whose $\mu$-graphs are $K_{t\times n}$ $(n\geqslant 2)$.
Juri\v{s}i\'{c}, Munemasa, and Tagami \cite{JMT2010} investigated a more general case, namely, graphs (not necessarily distance-regular) whose $\mu$-graphs are $K_{t\times n}$.
Moreover, several studies examined distance-regular graphs whose $\mu$-graphs are complete multipartite; see \cite{JK2000-2,JK2011, KLLLLT2023+}.
These studies have contributed to the research on classifying $1$-homogeneous distance-regular graphs with complete multipartite $\mu$-graphs, which is an important problem.

\smallskip
In this paper, one significance of our result lies in making substantial progress towards a classification of $1$-homogeneous distance-regular graphs with $a_1 > 0$. 
This result extends to a broader context, covering $1$-homogeneous graphs whose $\mu$-graphs are complete multipartite, as discussed in the preceding paragraph.
We now present the main result of this paper.

\begin{theorem}\label{thm:main}
Let $\Gamma$ be a $1$-homogeneous distance-regular graph with diameter $D \geqslant 5$ and $a_1 > 0$. 
Define $b = {b_1}/({\theta_1 + 1})$. 
Then, either $c_2=1$, or $b\geqslant 1$ and one of the following holds:
\begin{enumerate}[\normalfont(i)]
	\item $\Gamma$ is a regular near $2D$-gon.
	\item $\Gamma$ is a Johnson graph $J(2D,D)$.
	\item $\Gamma$ is a halved $\ell$-cube with $\ell \in \{2D,2D+1\}$.
	\item $\Gamma$ is a folded Johnson graph $\bar{J}(4D,2D)$.
	\item $\Gamma$ is a folded halved $4D$-cube.
	\item The valency $k$ of $\Gamma$ is bounded by a function $F(b)$ of $b$, {\rm i.e.}, $k \leqslant F(b)$, where
	\begin{equation}\label{poly: F}
	F(b) = 16b^{10} + 80b^9 + 192b^8 + 256b^7 + 192b^6 + 72b^5 + 20b^4 + 24b^3 + 8b^2 + 1. 
	\end{equation}
\end{enumerate}
\end{theorem}
The proof of this theorem appears in Section \ref{sec:1-homoDRGs}.

\begin{remark} 
(i) We comment on the cases where the diameter is $3$ or $4$ for $1$-homogeneous distance-regular graphs with $a_1 > 0$.
For $D = 3$, the Taylor graphs are examples of $1$-homogeneous graphs with $a_1 > 0$, as they are tight distance-regular graphs. 
For $D = 4$, examples include the Patterson graph and the family $AT_4(p, q, r)$ of antipodal tight graphs with parameters $p, q, r$; see \cite{JK2002, JK2011}. 
It remains an open question whether there exist infinitely many such graphs of diameter $4$.

(ii) In the case where $\Gamma$ is a regular near $2D$-gon as in Theorem~\ref{thm:main}, we can further refine the classification under additional conditions: 
if $c_2 \geqslant 3$, then $\Gamma$ is a dual polar graph; 
if $c_2 = 2$ and $c_3 = 3$, then $\Gamma$ is a Hamming graph; see \cite[Theorem~9.11]{DaKoTa}.
\end{remark}

\begin{remark}
In \cite{KLLLLT2023+}, Koolen et al.\ proposed a conjecture stating that for a tight distance-regular graph with $D \geqslant 3$ and $b = b_1 / (1 + \theta_1) \geqslant 2$, the diameter $D$ is bounded by a function of $b$; see \cite[Conjecture 7.5]{KLLLLT2023+}. We prove this conjecture in Section~\ref{sec:tightDRGs} using Theorem~\ref{thm:main}.
\end{remark}

This paper is organized as follows.
In Section \ref{sec:SRG}, we review strongly regular graphs and their properties. 
We discuss a classification of strongly regular graphs with smallest eigenvalue $\leqslant -2$.
In Section \ref{sec:DRG locally SRG}, we discuss distance-regular graphs that are locally strongly regular. 
We establish a bound on the intersection number $c_2$ for those graphs. 
We also show that when a distance-regular graph is locally a conference graph, it is a Taylor graph.
In Section \ref{Sec:CABprop}, we recall the $\operatorname{CAB}_i$ property of distance-regular graphs. 
We focus on the $\operatorname{CAB}_2$ property and examine distance-regular graphs that possess this property.
In Section \ref{sec:1-homoDRGs}, we discuss $1$-homogeneous distance-regular graphs with $a_1 > 0$. 
We prove our main result, Theorem \ref{thm:main}.
In Section \ref{sec:DRG-cp}, we discuss $1$-homogeneous distance-regular graphs with classical parameters and $a_1 > 0$.
Finally, we conclude the paper in Section \ref{sec:tightDRGs} with some comments on tight distance-regular graphs.

%%%%%%%%%%%%%%%%%%%%%%%%%%%%%%%%%%%%%%%%%%
%%%%%%%%%%%%%%%%%%%%%%%%%%%%%%%%%%%%%%%%%%
%%%%%%%%%%%%%%%%%%%%%%%%%%%%%%%%%%%%%%%%%%
\section{Strongly regular graphs with smallest eigenvalue $-m$}\label{sec:SRG}
In this section, we review properties of strongly regular graphs with smallest eigenvalue $-m$, where $m > 0$, and discuss their classification.
First, we recall the definition of a strongly regular graph.
Let $\Gamma$ be a $k$-regular graph with $v$ vertices. 
The graph $\Gamma$ is called \emph{strongly regular} with parameters $(v,k,\lambda, \mu)$ if each pair of distinct adjacent (resp. non-adjacent) vertices has exactly $\lambda$ (resp. $\mu$) common neighbors. 
Suppose $\Gamma$ is a strongly regular graph with smallest eigenvalue $s$. 
It is well known that $\Gamma$ satisfies $s \leqslant -2$, except in the cases where $\Gamma$ is a disjoint union of cliques (with $s = -1$) or a pentagon (with $s = (-1 - \sqrt{5})/{2}$) \cite[Section 1.1.10]{BrVM2022}.

\smallskip
Let $\Gamma$ be a strongly regular graph with parameters $(v,k,\lambda,\mu)$ and diameter two.
We denote the eigenvalues of $\Gamma$ as $k>r>s$. 
It is known that $k,r,s$ are integers except when $\Gamma$ is a \emph{conference} graph, i.e., a strongly regular graph with parameters $(4\mu +1, 2\mu, \mu-1,\mu)$ \cite[Lemma 10.3.3]{GR01}. 
The parameters $v, k, \lambda$ of $\Gamma$ can be expressed in terms of $r$, $s$, and $\mu$ as follows:
\begin{equation}\label{SRG parameters}
	  v = \frac{(k-r)(k-s)}{\mu}, \qquad 
	  k= \mu -rs, \qquad 
	  \lambda = \mu +r +s,
\end{equation}
cf. \cite[Theorem 1.3.1]{BCN}.
We present two examples of strongly regular graphs that will be used in this paper.

\begin{example}
A transversal design $\operatorname{TD}(m;n)$ is a partial linear space with $mn$ points and $m + n^2$ lines, with $m$ lines (called groups) of size $n$ forming a partition of the point set, and $n^2$ lines (called blocks) of size $m$, each meeting every group in a single point; cf. \cite[Section 8.4.1]{BrVM2022}.
The line graph of a transversal design $\operatorname{TD}(m;n)$ with $2 \leqslant m \leqslant n$ is called a \emph{Latin square} graph $\operatorname{LS}_m(n)$; cf. \cite[Section 8.4.2]{BrVM2022}.
Note that $\operatorname{LS}_m(n)$ is isomorphic to the block graph of an orthogonal array $\operatorname{OA}(m,n)$.
A Latin square graph $\operatorname{LS}_m(n)$ is strongly regular with parameters
\begin{equation}\label{eq:LS parameters}
	(n^2, \quad m(n-1), \quad (m-1)(m-2)+n-2, \quad m(m-1))
\end{equation}
and eigenvalues $m(n-1)>n-m>-m$.
\end{example}

\begin{example}
A Steiner system $S(2,m,n)$ is a $2$-$(n,m,1)$ design, that is, a collection of $m$-subsets of a $n$-set in which each pair of elements is contained in exactly one $m$-set.
In this context, the elements of the $n$-set are referred to as points, and the $m$-sets are referred to as blocks of the system.
The \emph{block graph of a Steiner system} $S(2, m, n)$ is defined as the graph whose vertices are the blocks of the system, where two vertices are adjacent whenever they intersect at exactly one point.
The block graph of a Steiner system $S(2, m, n)$ with $n>m\geqslant 2$ is strongly regular with parameters
\begin{equation}
	\left(\frac{n(n-1)}{m(m-1)}, \quad \frac{m(n-m)}{m-1}, \quad (m-1)^2 +\frac{n-1}{m-1} -2, \quad m^2\right).
\end{equation}
The eigenvalues of this graph are $\frac{m(n-m)}{m-1} > \frac{n-m^2}{m-1} > -m$.
In particular, the block graph of a Steiner system $S(2, m, mn + m -n)$ is called a \emph{Steiner graph} $S_m(n)$. 
Note that a Steiner graph $S_m(n)$ has parameters
\begin{equation}\label{eq:SG parameters}
	\left( \frac{(m+n(m-1))(n+1)}{m}, \quad mn, \quad m^2-2m+n, \quad m^2 \right)
\end{equation}
and eigenvalues $mn>n-m>-m$.
\end{example}

Next, we recall some known results on the classification of strongly regular graphs whose smallest eigenvalue is at most $-2$.
For the rest of this section, let $\Gamma$ be a strongly regular graph with parameters $(v, k, \lambda, \mu)$ and integral eigenvalues $k > r > s$.
For our purposes, we set 
\begin{equation}\label{Def:m,n}
	m := -s, \qquad n = r - s  \qquad \qquad (m \geqslant 2).
\end{equation}
If $\Gamma$ is primitive, that is, both $\Gamma$ and its complement are connected, the parameter $\mu$ is bounded above by a function of $m$:
\begin{equation}\label{eq:mu-bound}
	\mu \leqslant m^3 (2m -3).
\end{equation}
We call \eqref{eq:mu-bound} the $\mu$-bound; see \cite[Theorem 3.1]{Neumaier.1979}.
Let $f(m,\mu) = \frac{1}{2}m(m-1)(\mu+1) +m-1$.
Then, by \cite[Theorem 4.7]{Neumaier.1979} (cf. \cite[Theorem 8.6.3]{BrVM2022}), the following statements (i)--(iii) hold:
\begin{enumerate}[\normalfont(i)]
	\item If $\mu=m(m-1)$ and $n > f(m,\mu)$, then $\Gamma$ is a Latin square graph $\operatorname{LS}_m(n)$.
	\item If $\mu=m^2$ and $n > f(m,\mu)$, then $\Gamma$ is a Steiner graph $S_m(n)$.
	\item If $\mu\neq m(m-1)$ and $\mu \neq m^2$, then 
	\begin{equation}\label{eq:claw bound}
	n \leqslant f(m,\mu) = \frac{1}{2}m(m-1)(\mu+1) +m-1.
\end{equation} 
\end{enumerate}
We call \eqref{eq:claw bound} the \emph{claw bound}.
As a consequence of the $\mu$-bound and the claw bound, the strongly regular graphs with integral smallest eigenvalue $\leqslant -2$ are characterized as follows.

\begin{lemma}[Sims,~cf.~{\cite[Theorem 8.6.4]{BrVM2022}}]\label{SRG of Sims}
Let $\Gamma$ be a strongly regular graph with parameters $(v,k,\lambda, \mu)$ with integral smallest eigenvalue $-m$, where $m\geqslant 2$. 
Then $\Gamma$ belongs to one of the following {\rm(i)}--{\rm(iv)}:
\begin{enumerate}[\normalfont(i)]
  \item complete multipartite graphs with classes of size $m$;
  \item Latin square graphs $\operatorname{LS}_m(n)$;
  \item Steiner graphs $S_m(n)$;
  \item finitely many further graphs.
\end{enumerate}
\end{lemma}

\medskip
We give a comment on Lemma \ref{SRG of Sims}.
Assume that $\Gamma$ is none of a complete multipartite graph, a Latin square graph, or a Steiner graph.

By Lemma \ref{SRG of Sims} and the claw bound \eqref{eq:claw bound}, $\Gamma$ satisfies $n \leqslant \frac{1}{2}m(m-1)(\mu+1) +m-1$.
Since $n=r+m$, it follows 
\begin{equation}\label{eq: bound r}
	r\leqslant \frac{1}{2} m(m-1)(\mu+1)-1.
\end{equation}
Note that $n \ne m$ since $\Gamma$ is not a complete multipartite graph.
This implies $r \ne 0$, that is, $r \geqslant 1$.
From the first equation in \eqref{SRG parameters}, we have $v = (k-r)(k+m)/\mu$. 
Substitute $k$ with $\mu + rm$ and simplify the result to obtain
\begin{equation}\label{eq: formula v;r,m}
	v = \mu  + m - r + 2rm + \frac{rm(m-1)(1+r)}{\mu}. 
\end{equation}
Applying inequalities \eqref{eq:mu-bound} and \eqref{eq: bound r} to the right-hand side of \eqref{eq: formula v;r,m} and expressing the result in terms of $m$, we obtain
\begin{equation}\label{eq:bound v;m}
	 v \leqslant m^3(2m-3) + m + \left(2m-1+m^2(m-1)^2\right)\left(\frac{m(m-1)}{2}\left(m^3(2m-3)+1\right)-1\right).
\end{equation}
Let $\varphi(m)$ denote the right-hand side of \eqref{eq:bound v;m}. 
Simplify the expression for $\varphi(m)$ to obtain
\begin{equation}\label{eq:varphi(m)}
	\varphi(m) =  m^{10} - \frac{9}{2}m^9 +\frac{15}{2}m^8 - \frac{7}{2}m^7 - 4m^6 +4m^5 + m^4 - \frac{1}{2}m^3 - \frac{5}{2}m^2-\frac{1}{2}m + 1.
\end{equation}
Note that $\varphi(m) < m^{10}$ for all $m\geqslant 2$.
By these comments, we restate Lemma \ref{SRG of Sims} as follows:
%%%
\begin{corollary}\label{bound on order of SRG}
Let $\Gamma$ be a strongly regular graph with parameters $(v,k,\lambda, \mu)$ with integral smallest eigenvalue $-m$, where $m\geqslant 2$. 
Then one of the following holds:
\begin{enumerate}[\normalfont(i)]
  \item $\Gamma$ is a complete multipartite graph with class of size $m$,
  \item $\Gamma$ is a Latin square graph $\operatorname{LS}_m(n)$,
  \item $\Gamma$ is a Steiner graph $S_m(n)$,
  \item The number of vertices of $\Gamma$ is bounded by a function in $m$, i.e., $v \leqslant \varphi(m)$, where $\varphi(m)$ is from \eqref{eq:varphi(m)}.
\end{enumerate}
\end{corollary}

%%%%%%%%%%%%%%%%%%%%%%%%%%%%%%%%%%%%%%
%%%%%%%%%%%%%%%%%%%%%%%%%%%%%%%%%%%%%%
%%%%%%%%%%%%%%%%%%%%%%%%%%%%%%%%%%%%%%
%%%%%%%%%%%%%%%%%%%%%%%%%%%%%%%%%%%%%%
\section{Distance-regular graphs that are locally strongly regular}\label{sec:DRG locally SRG}
In this section, we discuss distance-regular graphs whose local graphs are strongly regular.
For such graphs, we give a bound on their intersection number $c_2$.
We also show that if a distance-regular graph is locally a conference graph, then it is a Taylor graph.
We begin by recalling some known results about distance-regular graphs that will be used, along with references for further discussion.

\begin{lemma}[{cf.~\cite[Theorem 4.4.3]{BCN}}]\label{lem:local min eig}
Let $\Gamma$ be a distance-regular graph of diameter $D\geqslant 3$ with eigenvalues $k=\theta_0 > \theta_1 >\cdots > \theta_D$ and intersection number $b_1$.
Let $b = b_1/(\theta_1+1)$.
Then, $b >0$.
Moreover, for each vertex $x$ in $\Gamma$, its local graph has the smallest eigenvalue $\geqslant -1-b$.
\end{lemma}

Let $\Gamma$ be a distance-regular graph with diameter at least two.
Recall the definition of the $\mu$-graph of $\Gamma$.
Observe that each $\mu$-graph has $c_2$ vertices and that all are isomorphic. 
For such $\Gamma$, the $\mu$-graph is also called the \emph{$c_2$-graph} of $\Gamma$.
If each $c_2$-graph is a regular graph with valency $\kappa$, then we say that $\Gamma$ is \emph{$c_2$-graph-regular} with parameter $\kappa$.
The graph $\Gamma$ is called a \emph{Terwilliger graph} when it is $c_2$-graph-regular and every $c_2$-graph of $\Gamma$ is complete.

\begin{lemma}[{cf.~\cite[Theorem 3.1]{JK2000-2}}]\label{mu-regular}
Let $\Gamma$ be a distance-regular graph which is locally strongly regular with parameters $(v',k',\lambda', \mu')$. 
Then $\Gamma$ is $c_2$-graph-regular with parameter $\mu'$.
Moreover, $c_2 \geqslant \mu'+1$, with equality if and only if $\Gamma$ is a Terwilliger graph.
\end{lemma}

A \emph{clique} in $\Gamma$ is a subset of $V(\Gamma)$ such that every pair of distinct vertices is adjacent. A clique of size $p$ is referred to as a complete graph $K_p$.
A \emph{coclique} of $\Gamma$ is a subset of $V(\Gamma)$ such that no two vertices are adjacent. 
A \emph{complete bipartite graph} $K_{p,q}$ is a graph whose vertex set can be partitioned into two cocliques, say a $p$-set $V_1$ and a $q$-set $V_2$, where each vertex in $V_1$ is adjacent to all vertices in $V_2$.
A \emph{complete multipartite graph} $K_{t \times p}$ is a graph whose vertex set can be partitioned into cocliques $\{V_i\}^t_{i=1}$ of size $p$, where each vertex in $V_i$ is adjacent to all vertices in $V_j$ $(1\leq j\ne i \leq t)$.

\begin{lemma}[{cf.~\cite[Lemma 3.7]{TKCP2022}}]\label{bound on s}
Let $\Gamma$ be a distance-regular graph with valency $k$, diameter $D \geqslant 5$ and second largest eigenvalue $\theta_1$. Assume that $\Gamma$ contains an induced subgraph $K_{2,t}$ for some $t \geqslant 2$. Let $b = {b_1}/(\theta_1 +1)$. Then $t \leqslant 4b^2 +1$.
\end{lemma}

We now establish a bound on the intersection number $c_2$ for distance-regular graphs that are locally strongly regular. 

\begin{proposition}
Let $\Gamma$ be a distance-regular graph with diameter $D\geqslant 5$, valency $k$, intersection numbers $b_1$ and $c_2$, and second largest eigenvalue $\theta_1$. 
Assume $\Gamma$ is locally strongly regular with parameters $(v',k',\lambda', \mu')$. 
Let $b = {b_1}/(\theta_1 +1)$. 
Then 
\begin{equation}\label{c2:bound}
	c_2 \leqslant (4b^2 +1)(\mu'+1).
\end{equation}
\end{proposition}
\begin{proof}
We assume $c_2\geqslant 2$; otherwise, it is trivial.
Suppose that $\Gamma$ is a Terwilliger graph.
Then, by Lemma \ref{mu-regular}, we have $c_2=\mu'+1 < (4b^2+1)(\mu'+1)$.
Now, assume that $\Gamma$ is not a Terwilliger graph.
Since $c_2\geqslant 2$, $\Gamma$ contains an induced subgraph $K_{2,t}$ for some $t\geqslant 2$.
Applying Lemma \ref{bound on s} to $\Gamma$, we obtain the bound $t \leqslant 4b^2+1$.
Next, as $\Gamma$ is locally strongly regular, by Lemma \ref{mu-regular}, $\Gamma$ is $c_2$-graph-regular with parameter $\mu'$.
Thus, we find that the number of vertices of the $c_2$-graph is at most $t(\mu'+1)$.
By these comments, it follows that $c_2 \leqslant t(\mu'+1) \leqslant (4b^2+1)(\mu'+1)$.
\end{proof}

Recall that a conference graph is a strongly regular graph with parameters $(4\mu +1, 2\mu, \mu-1, \mu)$, where $\mu > 0$. 
We consider distance-regular graphs that are locally a conference graph. 
First, we recall the classification of distance-regular graphs that satisfy $a_1 \geqslant k/2-1$, as shown by Koolen and Park \cite{KP2012}.

\begin{lemma}[{\cite[Theorem 16]{KP2012}}]\label{a_1 bound}
Let $\Gamma$ be a distance-regular graph with diameter $D \geqslant 3$ and valency $k$. If $a_1 \geqslant \frac{1}{2}k -1$, then one of the following holds:
\begin{enumerate}[\normalfont(i)]
  \item $\Gamma$ is a polygon,
  \item $\Gamma$ is the line graph of a Moore graph,
  \item $\Gamma$ is the flag graph of a regular generalized $D$-gon of order $(s,s)$ for some $s$,
  \item $\Gamma$ is a Taylor graph,
  \item $\Gamma$ is the Johnson graph $J(7,3)$,
  \item $\Gamma$ is the halved $7$-cube.
\end{enumerate}
\end{lemma}

\begin{proposition}\label{conference}
Let $\Gamma$ be a distance-regular graph with diameter $D \geqslant 3$. 
If $\Gamma$ is locally a conference graph, then it is a Taylor graph.
\end{proposition}
\begin{proof}
Let $\Delta$ denote a local graph of $\Gamma$. 
Then $\Delta$ is a conference graph, and we denote its parameters by $(4\mu +1, 2\mu, \mu-1,\mu)$, where $\mu > 0$.
Thus, $\Gamma$ has valency $k=4\mu+1$ and intersection number $a_1=2\mu$.
We observe that $k$ is odd, $a_1$ is nonzero and even, and $\Gamma$ satisfies $a_1 > \frac{1}{2}k -1$.
By Lemma \ref{a_1 bound}, $\Gamma$ falls into one of the graphs (i)--(vi) listed therein. 
However, we observe that:
\begin{itemize}\setlength\itemsep{0pt}
\item In case (i), a polygon has $k=2$ and $a_1=0$;
\item In case (ii), the line graph of a Moore graph has the intersection array $\{2k'-2,k'-1,k'-2;1,1,4\}$, where $k' \in \{3,7,57\}$.
This implies that $a_1 = k'-2$, which is an odd number;
\item In case (iii), the flag graph of a regular generalized $D$-gon of order $(s,s)$, a generalized $2D$-gon of order $(s,1)$, has $k= 2s$ (an even number);
\item In case (v), the Johnson graph $J(7,3)$ has the intersection array $\{12,6,2;1,4,9\}$, so $k=12$ (even) and $a_1=5$ (odd).
\end{itemize}
From these observations, it follows that none of the graphs (i), (ii), (iii), (v) is locally a conference graph.
Moreover, as the halved $7$-cube has the intersection array $\{21,10,3; 1,6,15\}$ its local graph has $21$ vertices and valency $10$, but such a conference graph does not exist.
Therefore, $\Gamma$ is a Taylor graph.
\end{proof}

We finish this section with a comment.
The line graph of $K_{p,q}$ is called the \emph{$(p\times q)$-grid}.
Note that the square $(p\times p)$-grid, where $p\geqslant 2$, is a strongly regular graph and is isomorphic to the Latin square graph $\operatorname{LS}_2(p)$.

\begin{lemma}[{cf. \cite[Proposition 4.1]{JK2003}}]\label{lem:grid}
Let $\Gamma$ be a distance-regular graph with diameter $D\geqslant 2$.
If $\Gamma$ is locally the $(p\times q)$-grid and $c_2=4$, then $\Gamma$ is the Johnson graph $J(p+q,p)$, or $p=q$ and $\Gamma$ is the folded Johnson graph $\bar{J}(2p,p)$.
\end{lemma}

%%%%%%%%%%%%%%%%%%%%%%%%%%%%%%%%%%%%%%%%%%
%%%%%%%%%%%%%%%%%%%%%%%%%%%%%%%%%%%%%%%%%%
%%%%%%%%%%%%%%%%%%%%%%%%%%%%%%%%%%%%%%%%%%
\section{The $\operatorname{CAB}_i$ property}\label{Sec:CABprop}

In this section, we recall the definition of the $\operatorname{CAB}_i$ property for distance-regular graphs and review related known results.  
We then focus on distance-regular graphs satisfying the $\operatorname{CAB}_2$ property.

\smallskip
Let $\Gamma$ be a distance-regular graph with diameter $D \geqslant 2$. 
For $0 \leqslant i \leqslant D$ and for $x, y \in V(\Gamma)$ with $d(x, y) = i$, we recall the subsets $C_i(x, y)$, $A_i(x, y)$, $B_i(x, y)$ of $V(\Gamma)$, as defined in \eqref{sets:CAB}. 
We also recall the intersection numbers $c_i$, $a_i$, $b_i$ of $\Gamma$ from \eqref{int numbers}.
Assume $a_1 > 0$.
Let $\Delta(y)$ denote the local graph of $\Gamma$ at $y$.
Note that $\Delta(y)$ has $k = b_0$ vertices and is regular with valency $a_1$.
We observe that the set $\{C_i(x,y), A_i(x,y), B_i(x,y)\}$ partitions the vertex set of $\Delta(y)$.
We denote this partition by $\operatorname{CAB}_i(x, y)$ and call it the \emph{$\operatorname{CAB}_i(x, y)$ partition} of $\Delta(y)$. 
Observe that the $\operatorname{CAB}_0(x,y)$ partition is $\{\Gamma(x)\}$, and the $\operatorname{CAB}_D(x,y)$ partition is $\{C_D(x,y), A_D(x,y)\}$ if $a_D\ne 0$ and $\{C_D(x,y)\}$ if $a_D=0$. 
For $1 \leqslant j \leqslant D$, we say that $\Gamma$ has the \emph{$\operatorname{CAB}_j$ property} if, for each $i \leqslant j$ and for every pair of vertices $x, y \in V(\Gamma)$ with $d(x,y)=i$, the partition $\operatorname{CAB}_i=\operatorname{CAB}_i(x,y)$ is equitable, and its parameters do not depend on the choice of vertices $x$ and $y$.
If $\Gamma$ has the $\operatorname{CAB}_D$ property, then we simply say that it has the \emph{CAB property}.
\begin{lemma}[{cf. \cite[Proposition 2.1]{JK2000-1}}]\label{CAB_1}
Let $\Gamma$ be a distance-regular graph with $a_1 > 0$ and the $\operatorname{CAB}_1$ property. 
Then all local graphs of $\Gamma$ are either
\begin{itemize}
  \item[\rm (i)] connected strongly regular graphs with the same parameters, or
  \item[\rm (ii)] disjoint unions of $(a_1 +1)$-cliques.
\end{itemize}
\end{lemma}

We say the (triple) intersection number $\gamma$ of $\Gamma$ \emph{exists} if, for every triple of vertices $(x,y,z)$ of $\Gamma$ such that $x$ and $y$ are adjacent and $z$ is at distance $2$ from both $x$ and $y$, the number of common neighbors of $x$, $y$, and $z$ is constant and equal to $\gamma$.
To avoid the degenerate case, we assume that there exists at least one such triple $(x,y,z)$ in $\Gamma$ (i.e., $a_2\neq0$) when we say $\gamma$ exists.
We note that if $\Gamma$ has the $\operatorname{CAB}_2$ property and $a_2\neq 0$, then the intersection number $\gamma$ exists in $\Gamma$. 
In this case, the parameter $\alpha_2$, defined as the number of neighbors in $C_2(x,y)$ of a vertex in $A_2(x,y)$, is equal to $\gamma$ (see Figure~\ref{fig:1}).

We recall the definition of a regular near polygon.
Let $\Gamma$ be a distance-regular graph with diameter $D$.
For integers $s$ and $t$, we say that $\Gamma$ is \emph{of order $(s,t)$} if it is locally the disjoint union of $t+1$ cliques of size $s$.
The graph $\Gamma$ of order $(s,t)$ is called a \emph{regular near polygon} if $a_i=c_ia_1$ for all $1\leq i \leq D-1$.
If $a_D=c_Da_1$, we call $\Gamma$ a \emph{regular near $2D$-gon}; otherwise it is called a \emph{regular near $(2D+1)$-gon}.
We note that if a distance-regular graph with diameter $D \geqslant 2$ and $a_1 > 0$ is locally disconnected and has the $\operatorname{CAB}$ property, then it is a regular near $2D$-gon; see \cite[Theorem 2.3]{JK2000-1}.

\begin{lemma}[{cf. \cite[Theorem 9.11]{DaKoTa}}]\label{dual polar}
Let $\Gamma$ be a regular near $2D$-gon with $D \geqslant 4$. 
If $c_2 \geqslant 3$ or $c_i = i$ $(i=2,3)$, then $\Gamma$ is either a dual polar graph or a Hamming graph.
\end{lemma}

Let $\Gamma$ be a distance-regular graph with diameter $D\geqslant 2$ and $a_1>0$.
Suppose $\Gamma$ has the $\operatorname{CAB}_j$ property for some $1\leq j \leq D$.
For each $1\leqslant i \leqslant j$, $i\neq D$, consider the parameters of the $\operatorname{CAB}_i$ partition (see Figure \ref{fig:1}).
Using these parameters, we define the matrix $Q_i$ by
\begin{equation}\label{Qi}
	Q_i =
	\begin{pmatrix}
	\gamma_i & a_1-\gamma_i & 0 \\
	\alpha_i & a_1 -\beta_i -\alpha_i & \beta_i \\ 
	0 & \delta_i & a_1 -\delta_i
	\end{pmatrix}
	\qquad \qquad 
	(1 \leqslant i \leqslant j, \ i \neq D).
\end{equation}
We call $Q_i$ the \emph{quotient matrix} associated with the $\operatorname{CAB}_i$ partition; cf. {\cite[Theorem 2.4]{JK2000-1}}.
Since $\Gamma$ has the $\operatorname{CAB}_j$ property, it follows from construction and Lemma \ref{CAB_1} that $\Gamma$ is locally strongly regular with the same parameters.
If $\Gamma$ is locally connected, then the eigenvalues of $Q_i$ coincide with those of a local graph of $\Gamma$; cf. \cite[Lemma 2.6]{JK2000-1}.

\begin{figure}
\centering
\begin{tikzpicture}%[baseline=0pt]
	\draw (1,0) -- (3,0);\draw (4,0) -- (6,0); % two line segments
	\draw (0.5,0) circle [radius=0.5]; % 1st circle
	\draw (3.5,0) circle [radius=0.5]; % 2nd circle
	\draw (6.5,0) circle [radius=0.5]; % 3rd circle
	\node at (0.5,0) {$c_i$};  
	\node at (3.5,0) {$a_i$};
	\node at (6.5,0) {$b_i$};
	\node [above] at (0.5,0.5) {$C_i$}; 
	\node [below] at (0.5,-0.5) {$\gamma_i$}; 
	\node [below] at (1.6,0) {$a_1-\gamma_i$};
	\node [below] at (2.8,0) {$\alpha_i$}; 
	\node [above] at (3.5,0.5) {$A_i$}; 	
	\node [below] at (3.5,-0.5) {$a_1 -\alpha_i -\beta_i$}; 
	\node [below] at (4.2,0) {$\beta_i$};
	\node [below] at (5.8,0) {$\delta_i$}; 
	\node [above] at (6.5,0.5) {$B_i$}; 
	\node [below] at (6.5,-0.5) {$a_1 -\delta_i$};
\end{tikzpicture}
\caption{The $\operatorname{CAB}_i$ partition and its parameters for $1 \leqslant i \leqslant D-1$}\label{fig:1}
\end{figure}
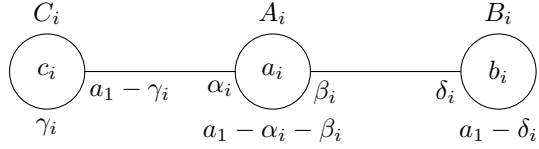

\begin{lemma}[cf. {\cite[Theorem 2.7]{JK2000-1}}] \label{lem:CAB parameters}
Let $\Gamma$ be a locally connected distance-regular graph with $D\geqslant 2$, $a_1\geqslant 2$, and the $\operatorname{CAB}_j$ property for some $1\leqslant j<D$. Let $\delta_0=0$.
Then, for $1\leqslant i \leqslant j$ the parameters $a_i$, $b_i$, $\alpha_i$, $\beta_i$, $\gamma_i$, $\delta_i$ of the $\operatorname{CAB}_i$ partition are expressed in terms of the eigenvalues $a_1>r>s$ of a local graph of $\Gamma$ and the parameters $\delta_{i-1}$ and $c_i$ as follows:
\begin{gather}
	\gamma_i  = \delta_{i-1}, \qquad
	\alpha_i = \frac{c_i(a_1-\delta_{i-1})}{k-c_i-b_i}, \qquad
	\beta_i = \frac{\mu'b_i}{a_1-\delta_{i-1}}, \qquad
	\delta_i = \frac{\mu'(k-c_i)}{a_1-\delta_{i-1}} - \beta_i, 	\label{formulas: parameters} \\
	b_i  = k - c_i - \frac{c_i(a_1 - \delta_{i-1})^2}{(a_1-\delta_{i-1})(a_1 - r - s + \delta_{i-1}) - \mu'(k-c_i)}, \label{formula: b_i}
\end{gather}
and $a_i = k - b_i - c_i$, where $k=(a_1-r)(a_1-s)/\mu'$ and $\mu'=a_1+rs$.
\end{lemma}
\begin{proof}
For \eqref{formulas: parameters}, refer to the proof of \cite[Theorem 2.7]{JK2000-1}.
We will show \eqref{formula: b_i} only.
Observe that $\Gamma$ has the $\operatorname{CAB}_1$ property, so it is locally a strongly regular graph with parameters $(k,a_1, \lambda', \mu')$ and eigenvalues $a_1 > r > s$.
Consider the quotient matrix $Q_i$ associated with the $\operatorname{CAB}_i$ partition for $1\leqslant i \leqslant j < D$.
From \eqref{Qi}, we have the trace $\operatorname{tr}(Q_i) = 2a_1 - \alpha_i - \beta_i - \delta_i + \gamma_i$.
Moreover, since the eigenvalues of $Q_i$ equal the eigenvalues of a local graph of $\Gamma$, it follows that $\operatorname{tr}(Q_i) = a_1 + r + s$.
From these two equations for the trace of $Q_i$, we obtain
\begin{equation}\label{eq:trace Qi}
	a_1 - r - s  = \alpha_i+\beta_i+\delta_i - \gamma_i .
\end{equation}
Eliminate $\alpha_i$, $\beta_i$, $\delta_i$, $\gamma_i$ in \eqref{eq:trace Qi} using \eqref{formulas: parameters} and solve for $b_i$. 
Simplify the result to obtain \eqref{formula: b_i}.
\end{proof}

Now, we discuss distance-regular graphs satisfying the $\operatorname{CAB}_2$ property.
Let $\Gamma$ be a locally connected distance-regular graph with diameter $D \geqslant 2$ and $a_1\geqslant 2$.
We assume that $\Gamma$ satisfies the $\operatorname{CAB}_2$ property. 
By Lemma \ref{CAB_1}, every local graph of $\Gamma$ is a connected strongly regular graph with the same parameters $(v', k', \lambda', \mu')$ and eigenvalues $a_1 > r > s$.
Note that $v'=k$ and $k'=a_1$.
Consider the $\operatorname{CAB}_2$ partition. 
By construction, we have $\delta_1=\mu'$.
Using this and $\lambda' = \mu' + r + s$, evaluate equations \eqref{formulas: parameters}, \eqref{formula: b_i} at $i=2$. 
Then the parameters of the $\operatorname{CAB}_2$ partition are 
\begin{equation}\label{constants2}
	\gamma_2 = \delta_{1}=\mu', \qquad
	\alpha_2 = \frac{c_2(a_1 - \mu')}{k-c_2-b_2}, \qquad 
	\beta_2 = \frac{\mu'b_2}{a_1-\mu'}, \qquad 
	\delta_2 = \frac{\mu'(k-c_2-b_2)}{a_1-\mu'},
\end{equation}
and
\begin{equation}\label{b2}
  b_2 = k-c_2 - \frac{c_2(a_1 -\mu')^2}{(a_1 +2\mu' -\lambda')(a_1 -\mu') -\mu'(k-c_2)}.
\end{equation}
With reference to the above discussion, we will express the parameters of the $\operatorname{CAB}_2$ partition in terms of $m=-s$ and $n=r-s$ when $\Gamma$ is locally a Latin square graph or a Steiner graph.
For our convenience, we will use the following notation.
\begin{notation}
Let $b \geqslant 1$. We define the polynomial $G$ in the variable $b$ as follows:
\begin{equation}\label{poly: G}
	G(b) 
	= 16^{10} + 32b^9 + 48b^8 + 32b^7 + 16b^6 + 8b^5 + 8b^4 + 8b^3 + 1.
\end{equation}
Recall $F(b)$ from \eqref{poly: F}.
Note that $G(b) < F(b)$ for all $b \geqslant 1$.
\end{notation}

\smallskip
Recall the intersection number $b_1$ and the second largest eigenvalue $\theta_1$ of $\Gamma$.

\begin{proposition}\label{LS}
Let $\Gamma$ be a distance-regular graph with diameter $D \geqslant 5$, valency $k$, $a_1 > 0$, and satisfying the $\operatorname{CAB}_2$ property. 
Let $b=b_1/(\theta_1+1)$.
Assume that $k > G(b)$, where $G(b)$ is from \eqref{poly: G}.
If $\Gamma$ is locally a Latin square graph $\operatorname{LS}_m(n)$ with $m \geqslant 2$,
then we have
\begin{align}
	&& & \alpha_2 =m, && \beta_2 =(m-1)(n-m^2+m), && \delta_2 = m^2(m-1), && \label{lem:LS eq(1)}\\
	&& & a_2 = m^2(n-m), && b_2 = (n-m)(n-m^2+m), && c_2 = m^2. && \label{lem:LS eq(2)}
\end{align}
\end{proposition}

\begin{proof} 
From \eqref{eq:LS parameters}, we observe that the parameters $(k,a_1,\lambda', \mu')$ of $\operatorname{LS}_m(n)$ are 
\begin{equation}\label{eq(1): pf of LS}
	k= n^2, \qquad a_1 =m(n-1), \qquad \lambda' = (m-1)(m-2) +n-2, \qquad \mu' = m(m-1).
\end{equation}
First, we find the parameters $\alpha_2$ and $c_2$.
In equation \eqref{eq:trace Qi} at $i=2$, eliminate $\beta_2$ and $\delta_2$ using \eqref{constants2}, then solve the resulting expression for $\alpha_2$ to obtain $\alpha_2 = a_1-\lambda'+2\mu' - {\mu'(k-c_2)}/({a_1-\mu'}).$
In this equation, eliminate the parameters $k, a_1, \lambda', \mu'$ using \eqref{eq(1): pf of LS} and express the result in terms of $n, m, c_2$ to have 
\begin{equation}\label{pf eq:lem LS}
	\alpha_2 = m + \frac{(m-1)(c_2-m^2)}{n-m}.
\end{equation}
Since $\alpha_2$ is an integer, we have either $(m-1)(c_2-m^2)=0$ or $n-m$ divides $(m-1)(c_2-m^2)$.
Suppose $(m-1)(c_2-m^2) \neq 0$.
Since $n-m$ divides $(m-1)(c_2-m^2)$, we have
\begin{equation}\label{pf eq(1):lem LS}
	n-m \leqslant (m-1)(c_2-m^2).
\end{equation}
Apply inequality \eqref{c2:bound} to $c_2$ in \eqref{pf eq(1):lem LS} and substitute $\mu'=m(m-1)$ into the result to get
\begin{equation}\label{pf eq(2):lem LS}
	n - m \leqslant (m-1) \big((4b^2+1)(m(m-1)+1) - m^2 \big).
\end{equation}
Simplify inequality \eqref{pf eq(2):lem LS} to get
\begin{equation}\label{pf eq(3):lem LS}
	n \leqslant 4b^2 (m-1)(m(m-1)+1) -m(m-1) +2m-1. 
\end{equation}
By Lemma \ref{lem:local min eig}, we have $m \leqslant 1+b$.
Apply this inequality to \eqref{pf eq(3):lem LS} to obtain
\begin{equation}\label{pf eq(4):lem LS}
	n \leqslant 4b^3 (b^2+b+1) -m(m-1) +2m-1. 
\end{equation} 
Moreover, since $m \geqslant 2$ and $m-1>0$, we have $-m(m-1) \leqslant -2m+2$.
Apply this inequality to \eqref{pf eq(4):lem LS} to obtain
\begin{equation}\label{pf eq(5):lem LS}
	n \leqslant 4b^3 (b^2+b+1) +1 = 4b^5 + 4b^4 + 4b^3 + 1. 
\end{equation} 
Recall that $k=n^2$ from \eqref{eq(1): pf of LS}.
Squaring both sides of \eqref{pf eq(5):lem LS} gives $k \leqslant (4b^5 + 4b^4 + 4b^3 + 1)^2=G(b)$.
This contradicts the assumption $k > G(b)$.
Therefore, we must have $(m-1)(c_2-m^2)=0$, which implies that $c_2=m^2$.
From \eqref{pf eq:lem LS}, we conclude that $\alpha_2=m$.

Next, we find $\beta_2$.
In the third equation in \eqref{constants2}, eliminate $b_2$ using \eqref{b2} to obtain
\begin{equation}\label{eq(3): pf of LS}
	\beta_2 = \frac{\mu'}{a_1-\mu'} \left( k-c_2 - \frac{c_2(a_1 -\mu')^2}{(a_1 +2\mu' -\lambda')(a_1 -\mu') -\mu'(k-c_2)}\right)
\end{equation} 
Express the right-hand side of \eqref{eq(3): pf of LS} in terms of $m$ and $n$ using \eqref{eq(1): pf of LS} and $c_2=m^2$.
Simplify the result to obtain $\beta_2 = (m-1)(n-m^2+m)$.
To find $b_2$, evaluate the right-hand side of \eqref{b2} using \eqref{eq(1): pf of LS} and $c_2=m^2$. 
Simplify the result to obtain $b_2=(n-m)(n-m^2+m)$.
For $\delta_2$, solve the fourth equation in \eqref{constants2} using \eqref{eq(1): pf of LS} and $b_2=(n-m)(n-m^2+m)$ to obtain $\delta_2 = m^2(m-1)$.
Finally, for $a_2$, evaluate $k-b_2-c_2$ using $k=n^2$, $b_2=(n-m)(n-m^2+m)$, $c=m^2$ to get $a_2=m^2(n-m)$. 
The proof is now complete.
\end{proof}

\begin{proposition}\label{Steiner}
Let $\Gamma$ be a distance-regular graph with diameter $D \geqslant 5$, valency $k$, $a_1 > 0$, and satisfying the $\operatorname{CAB}_2$ property. 
Let $b=b_1/(\theta_1+1)$.
Assume that $k > F(b)$, where $F(b)$ is from \eqref{poly: F}.
If $\Gamma$ is locally a Steiner graph $S_m(n)$ with $m \geqslant 2$, then we have
\begin{align}
	&& & \alpha_2 =m+1, && \beta_2 =(m-1)(n-m^2+1), && \delta_2 = m^3, &&  \label{lem:SG para(1)}\\
	&& & a_2 = m^2(n-m), && b_2 = (m-1)(n-m)(n-m^2+1)/m, && c_2 = m(m+1). && \label{lem:SG para(2)}
\end{align}
\end{proposition}
\begin{proof}
From \eqref{eq:SG parameters}, we observe that the parameters $(k,a_1,\lambda', \mu')$ of $S_m(n)$ are 
\begin{equation}\label{eq(1): pf of SG}
	k= \frac{(m+n(m-1))(n+1)}{m}, \qquad a_1 =mn, \qquad \lambda' =  m^2-2m+n, \qquad \mu' = m^2.
\end{equation} 
In a similar manner to the proof of Proposition \ref{LS}, we have
\begin{equation}\label{pf eq:alpha_2}
	\alpha_2 = m+1 + \frac{m(c_2 -m(m+1))}{n-m}.
\end{equation}
Since $\alpha_2$ is an integer, we have either $m(c_2 -m(m+1))=0$ or $n-m$ divides $m(c_2 -m(m+1))$.
If $c_2 \neq m(m+1)$, then 
\begin{equation}\label{pf eq(2):pf of SG}
	n-m \leqslant m(c_2 -m(m+1)).
\end{equation}
Apply inequality \eqref{c2:bound} to $c_2$ in \eqref{pf eq(2):pf of SG} and substitute $\mu' = m^2$ into the result to get
\begin{equation}\label{pf eq(3):pf of SG}
	n - m \leqslant m\big( (4b^2+1)(m^2+1) - m(m+1) \big).
\end{equation}
Simplify inequality \eqref{pf eq(3):pf of SG} to get
\begin{equation}\label{pf eq(4):pf of SG}
	n \leqslant m \big( 4b^2(m^2+1) - m +2 \big).
\end{equation}
Note that $2\leqslant m \leqslant 1+b$.
Apply this to \eqref{pf eq(4):pf of SG} to obtain
\begin{equation}\label{pf eq(5):pf of SG}
	n \leqslant m (4b^4 + 8b^3 + 8b^2).
\end{equation}
Recall that $k = (m+n(m-1))(n+1)/n$ from \eqref{eq(1): pf of SG}.
Apply inequality \eqref{pf eq(5):pf of SG} to this equation, along with $2\leqslant m \leqslant 1+b$, to obtain
\begin{equation}\label{pf eq(6):pf of SG}
	k = \left(\frac{m+n(m-1)}{m}\right)(n+1)\leqslant \big(1+b(4b^4 + 8b^3 + 8b^2)\big)\big(1+(1+b)(4b^4 + 8b^3 + 8b^2)\big) = F(b).
\end{equation}
This contradicts the given assumption that $k > F(b)$. 
Therefore, we must have $m(c_2-m(m+1))=0$, which implies that $c_2 = m(m+1)$.
From \eqref{pf eq:alpha_2}, we conclude that $\alpha_2 = m+1$.
The remaining parameters in \eqref{lem:SG para(1)} and \eqref{lem:SG para(2)} are obtained similarly to Proposition \ref{LS}. 
The proof is complete.
\end{proof}

\begin{remark}\label{rmk:Latin, Steiner}
(i) Let $\Gamma$ be locally a Latin square graph $LS_2(n)$ with $n \geqslant 3$.
Observe that $LS_2(n)$ has smallest eigenvalue $-2$ and is isomorphic to the square $(n \times n)$-grid.
By Proposition \ref{LS}, we have $c_2 = 4$. 
Therefore, by construction the $c_2$-graph of $\Gamma$ is a quadrangle.

\medskip
(ii) Let $\Gamma$ be locally a Steiner graph $S_2(n)$ with $n \geqslant 3$.
Observe that $S_2(n)$ has smallest eigenvalue $-2$ and is isomorphic to the triangular graph $T(n+2)$.
By Proposition \ref{Steiner}, we have $c_2 = 6$.
Therefore, by construction the $c_2$-graph of $\Gamma$ is a $4$-regular graph on $6$ vertices, i.e., an octahedron.
\end{remark}

%%%%%%%%%%%%%%%%%%%%%%%%%%%%%%%%%%%%%%%%%%
%%%%%%%%%%%%%%%%%%%%%%%%%%%%%%%%%%%%%%%%%%
%%%%%%%%%%%%%%%%%%%%%%%%%%%%%%%%%%%%%%%%%%
\section{The $1$-homogeneous distance-regular graphs}\label{sec:1-homoDRGs}
In this section, we discuss distance-regular graphs that have the $1$-homogeneous property. 
We then prove our main result, Theorem \ref{thm:main}.
We begin by recalling the definition of $i$-homogeneous graphs and some known results concerning these graphs.

Let $\Gamma$ be a connected graph with diameter $D$. 
For $0 \leqslant i \leqslant D$, fix a pair of vertices $x$ and $y$ in $\Gamma$ with $d(x,y)=i$.
For $0 \leqslant h, j \leqslant D$, we define the subset $D^h_j(x, y)$ of $V(\Gamma)$ by
\begin{equation}
	D^h_j(x,y) := \Gamma_j(x) \cap \Gamma_h(y).
\end{equation}
Let $\pi(x,y)$ denote the collection of nonempty sets $D^h_j(x,y)$ for $0\leqslant h,j \leqslant D$.
Observe that $\pi(x,y)$ forms a distance partition of $V(\Gamma)$.
The graph $\Gamma$ is said to be \emph{$i$-homogeneous} whenever the partition $\pi(x,y)$ is equitable, and the parameters of $\pi(x,y)$ are independent of the choice of vertices $x$ and $y$; see \cite{Nomura1994}.
Our discussion centers on the $1$-homogeneous distance-regular graphs.
We note that if $\Gamma$ is a $1$-homogeneous distance-regular graph with $D\geqslant 2$ and $a_2\neq 0$, then the intersection number $\gamma$ exists.
Juri\v si\' c and Koolen \cite{JK2000-1} studied these graphs in detail using the CAB property and characterized them as follows.

\begin{lemma}[{cf.~\cite[Theorem~3.1]{JK2000-1}}]\label{CAB_D}
Let $\Gamma$ be a distance-regular graph with diameter $D$ and $a_1 > 0$. 
Then $\Gamma$ is $1$-homogeneous if and only if it has the $\operatorname{CAB}$ property.
\end{lemma}

We consider $1$-homogeneous distance-regular graphs that are locally disconnected.
\begin{corollary}[{cf.~\cite[Corollary~3.3]{JK2000-1}}]\label{locally disconnected}
Let $\Gamma$ be a locally disconnected $1$-homogeneous distance-regular graph with diameter $D \geqslant 2$ and $a_1 > 0$.
Then $\Gamma$ is a regular near $2D$-gon.
\end{corollary}
\begin{proof}
Since $\Gamma$ is $1$-homogeneous, it has the $\operatorname{CAB}$ property by Lemma~\ref{CAB_D}.
As $\Gamma$ is also locally disconnected with $a_1 > 0$, it follows from \cite[Theorem 2.3]{JK2000-1} that $\Gamma$ is a regular near $2D$-gon.
\end{proof}

We have a comment.
Suppose that $\Gamma$ is a locally disconnected $1$-homogeneous distance-regular graph with $D \geqslant 4$ and $a_1 > 0$, and assume that either $c_2 \geqslant 3$ or $c_i = i$ for $i = 2,3$.
Then, by Corollary~\ref{locally disconnected}, $\Gamma$ is a regular near $2D$-gon. 
It follows from Lemma~\ref{dual polar} that $\Gamma$ is either a dual polar graph or a Hamming graph.

We now discuss locally connected $1$-homogeneous distance-regular graphs. 
There are some studies on $1$-homogeneous distance-regular graphs whose $c_2$-graphs are complete multipartite; cf. \cite{JK2003}, \cite{JMT2010}. 
In particular, Juri\v si\' c and Koolen \cite{JK2003} classified $1$-homogeneous distance-regular graphs whose $c_2$-graphs are the Cocktail Party graph.

\begin{theorem}[{\cite[Theorem 1.1]{JK2003}}]\label{1-homogenious JK}
Let $\Gamma$ be a $1$-homogeneous distance-regular graph with diameter $D \geqslant 2$, whose $c_2$-graph is a connected Cocktail Party graph. 
Then $\Gamma$ is one of the following graphs:
\begin{enumerate}[\normalfont(i)]
  \item a Johnson graph $J(2D,D)$,
  \item a halved $\ell$-cube with $\ell \in \{2D,2D+1\}$,
  \item a folded Johnson graph $\bar{J}(4D,2D)$,
  \item a folded halved $4D$-cube,
  \item a Cocktail Party graph $K_{t \times 2}$ with $t \geqslant 3$.
  \item the Schl\"{a}fli graph with intersection array $\{16,5;1,8\}$,
  \item the Gosset graph with intersection array $\{27,10,1; 1,10,27\}$.
\end{enumerate}
\end{theorem}

\smallskip
Recently, Koolen et al. \cite{KLLLLT2023+} studied $1$-homogeneous distance-regular graphs that are locally block graphs of Latin square graphs or Steiner systems. 
The main result they showed is as follows:
\begin{lemma}[cf. {\cite[Theorem 1.2]{KLLLLT2023+}}]\label{lower bound of k}
Let $\Gamma$ be a distance-regular graph with diameter $D \geqslant 3$, valency $k$, and $a_1>0$ satisfying the $\operatorname{CAB}_2$ property.
\begin{enumerate}[\normalfont(i)]
  \item If $\Gamma$ is locally a Latin square graph with smallest eigenvalue $-m$ for $m \geqslant 3$ and $k > m^2$, then $c_2 \neq m^2$.
  \item If $\Gamma$ is locally the block graph of a Steiner system $S(2,m,n)$ with smallest eigenvalue $-m$ for $m \geqslant 3$ and $k >m(m+1)$, then $c_2 \neq m(m+1)$. 
\end{enumerate}
\end{lemma} 

\begin{remark}
(i) Let $\Gamma$ be a distance-regular graph that is locally a Latin square graph $\operatorname{LS}_m(n)$ (resp. the block graph of a Steiner system $S(2,m,n)$). 
In \cite{KLLLLT2023+}, it is shown that if $c_2=m^2$ (resp. $c_2=m(m+1)$), then the $c_2$-graph of $\Gamma$ is a complete $m$-partite (resp. $(m+1)$-partite) graph. 
This result, along with the findings of \cite{JMT2010}, is used to complete the proof of Lemma \ref{lower bound of k}.

\medskip
(ii) The icosahedron has $b_1=2$, $\theta_1=\sqrt{5}$ and hence $b = \frac{b_1}{\theta_1 +1} =\frac{1}{2}(-1+\sqrt{5})< 1$.
The only locally connected distance-regular graph with the $\operatorname{CAB}_2$ property that satisfies $b  < 1$ is the icosahedron.
This is because the only connected strongly regular graph with the smallest eigenvalue greater than $-2$ is the pentagon, and the only graph that is locally a pentagon is the icosahedron (cf. \cite[p.~35]{BCN}).
\end{remark}

We are now ready to prove Theorem \ref{thm:main}.
\begin{proof}[Proof of Theorem \ref{thm:main}] 
We assume $c_2\geqslant 2$.
Our proof is divided into two cases: $\Gamma$ is locally disconnected, and $\Gamma$ is locally connected.
First, we assume that $\Gamma$ is locally disconnected.
Since $\Gamma$ contains a quadrangle, by the result of \cite{Terwilliger1985} (cf. \cite[p.~170]{BCN}), we obtain $\theta_1 \leqslant b_1-1$, which implies $b \geqslant 1$.
Moreover, by Corollary \ref{locally disconnected} $\Gamma$ is a regular near $2D$-gon. Therefore, we have (i).

\smallskip
Next, we consider the case where $\Gamma$ is locally connected.
Since $\Gamma$ satisfies the $\operatorname{CAB}_1$ property, every local graph of $\Gamma$ is connected strongly regular with the same parameters.
Let $-m$ be the smallest eigenvalue of a local graph of $\Gamma$. 
Note that a connected strongly regular graph satisfies $-m < -1$ and the only strongly regular graph satisfying $-m > -2$ is a pentagon.
Furthermore, the only graph that is locally pentagonal is the icosahedron.
However, an icosahedron has diameter $3$, which is less than $5$.
Therefore, we have $-m \leqslant -2$. 
By Lemma \ref{lem:local min eig}, it follows that $b\geqslant 1$.

\smallskip
Now, we show that one of the following (ii)--(vi) holds.
We assume that $k > F(b)$; otherwise, we fall into the case (vi).
Since $\Gamma$ is locally a connected strongly regular graph with eigenvalues $a_1>r>s$, it follows that either $\Gamma$ is locally a conference graph or $r$ and $s$ are integers.
If $\Gamma$ is locally a conference graph, then it must be a Taylor graph by Proposition \ref{conference}.
However, $\Gamma$ has $D\geqslant 5$ while a Taylor graph has diameter three. 
Therefore, a local graph of $\Gamma$ is not a conference graph and hence has integral eigenvalues $r$ and $s$. 
We put $s=-m$ and $r=n-m$, where $m\geqslant 2$.
According to Corollary \ref{bound on order of SRG}, a local graph of $\Gamma$ satisfies one of (i)--(iv) listed therein. 
However, Case (iv) of Corollary \ref{bound on order of SRG} is ruled out as $b \geqslant m-1 \geqslant 1$, and, by assumption, we have $k > F(b) > m^{10} > \varphi(m)$.
Therefore, $\Gamma$ is locally either a complete multipartite graph with classes of size $m$, a Latin square graph $\operatorname{LS}_m(n)$, or a Steiner graph $S_m(n)$.
Now, we consider each case.

\smallskip
First, consider the case where $\Gamma$ is locally a complete multipartite graph with classes of size $m$, denoted as $K_{t \times m}$. 
By \cite[Proposition 1.1.5]{BCN}, $\Gamma$ is $K_{(t+1) \times m}$. 
This contradicts the assumption $D \geqslant 5$ since a complete multipartite graph has diameter two. 	

\smallskip
Next, consider the case where $\Gamma$ is locally a Latin square graph $\operatorname{LS}_m(n)$.
Since $k > F(b) > G(b)$, by Proposition \ref{LS} we have $c_2=m^2$, where $m\geqslant 2$.
However, by Lemma \ref{lower bound of k}(i), $c_2$ cannot be equal to $m^2$ for $m\geqslant 3$.
By these comments, it follows $m=2$.
Thus, $\Gamma$ is locally a Latin square graph $\operatorname{LS}_2(n)$ with smallest eigenvalues $-2$ and $c_2=4$.
Observe that $\operatorname{LS}_2(n)$ is isomorphic to the $(n \times n)$-grid.
By Lemma \ref{lem:grid}, $\Gamma$ is the Johnson graph $J(2n,n)$ or a folded Johnson graph $\bar{J}(2n,n)$.

\smallskip
Lastly, consider the case where a local graph of $\Gamma$ is a Steiner graph $S_m(n)$.
Since $k > F(b)$, by Proposition \ref{Steiner} we have $c_2=m(m+1)$, where $m\geqslant 2$.
However, by Lemma \ref{lower bound of k}(ii) $c_2$ cannot be equal to $m(m+1)$ for $m\geqslant 3$.
Consequently, it follows $m=2$.
Thus, $\Gamma$ is locally a Steiner graph $S_2(n)$ with smallest eigenvalue $-2$ and $c_2 = 6$.
By Remark \ref{rmk:Latin, Steiner}(ii), the $c_2$-graph of $\Gamma$ is a $4$-regular graph on $6$ vertices, that is a Cocktail Party graph $K_{3\times2}$. 
Hence, by Theorem \ref{1-homogenious JK} (also refer to the proof of \cite[Theorem 5.3]{JK2003}) $\Gamma$ is a halved $\ell$-cube with $\ell \in \{2D,2D+1\}$ or a folded halved $4D$-cube.
The proof is complete.
\end{proof}

%%%%%%%%%%%%%%%%%%%%%%%%%%%%%%%%%%%%%%%%
%%%%%%%%%%%%%%%%%%%%%%%%%%%%%%%%%%%%%%%%
%%%%%%%%%%%%%%%%%%%%%%%%%%%%%%%%%%%%%%%%
\section{Distance-regular graphs with classical parameters}\label{sec:DRG-cp}
In this section, we discuss $1$-homogeneous distance-regular graphs with classical parameters.
For integers $b$ and $i$, we recall the Gaussian binomial coefficient
\begin{equation*}
	\genfrac{\lbrack}{\rbrack}{0pt}{}{i}{1}= \genfrac{\lbrack}{\rbrack}{0pt}{}{i}{1}_b 
	=\begin{cases} 
		i & \quad \text{if} \quad b = 1;\\
		(b^i-1)/(b-1) & \quad \text{if} \quad b\ne 1.
	\end{cases}
\end{equation*}
Let $\Gamma$ be a distance-regular graph with diameter $D \geqslant 3$.
We say that $\Gamma$ has \emph{classical parameters} $(D,b,\alpha, \beta)$ if its intersection numbers  $\{b_i\}^{D-1}_{i=0}$ and $\{c_i\}^D_{i=1}$ satisfy
\begin{equation}\label{clasic parameters}
	b_i = \left(\genfrac{\lbrack}{\rbrack}{0pt}{}{D}{1} -\genfrac{\lbrack}{\rbrack}{0pt}{}{i}{1}\right)\left(\beta -\alpha \genfrac{\lbrack}{\rbrack}{0pt}{}{i}{1}\right),
	\qquad \qquad
	c_i = \genfrac{\lbrack}{\rbrack}{0pt}{}{i}{1}\left(1 + \alpha \genfrac{\lbrack}{\rbrack}{0pt}{}{i-1}{1}\right).
\end{equation}
Note that the parameter $b$ is an integer, excluding $0$ or $-1$, since $b_2 \neq 0$ and $c_2 \neq 0$. 
Also, from $a_i + b_i + c_i = k = b_0$ $(0 \leqslant i \leqslant D)$, it follows
\begin{align}
	a_i & = \genfrac{\lbrack}{\rbrack}{0pt}{}{i}{1}\left( \beta-1+\alpha\left( \genfrac{\lbrack}{\rbrack}{0pt}{}{D}{1}-\genfrac{\lbrack}{\rbrack}{0pt}{}{i}{1} - \genfrac{\lbrack}{\rbrack}{0pt}{}{i-1}{1} \right) \right) \qquad (0\leqslant i \leqslant D). \label{eq:formula ai}
\end{align}
Observe that $a_1=0$ if and only if $\beta = 1-\alpha b \genfrac{\lbrack}{\rbrack}{0pt}{}{D-1}{1}$.
The eigenvalues of $\Gamma$ are 
\begin{equation}\label{eq:eig CP}
	\theta_i = \genfrac{\lbrack}{\rbrack}{0pt}{}{D-i}{1}\left( \beta - \alpha \genfrac{\lbrack}{\rbrack}{0pt}{}{i}{1} \right) - \genfrac{\lbrack}{\rbrack}{0pt}{}{i}{1} \qquad (0\leqslant i \leqslant D),
\end{equation}
cf. \cite[Corollary 8.4.2]{BCN}.
We note that the eigenvalues are in the natural ordering $\theta_0 > \theta_1> \cdots > \theta_D$ if $b>0$.
We recall a lower bound on the parameter $\beta$.
\begin{lemma}[{\cite[Proposition 1]{JV2017}}]\label{beta}
Let $\Gamma$ be a distance-regular graph with classical parameters $(D,b,\alpha, \beta)$ and $D \geqslant 3$. 
If $b > 0$, then $\beta \geqslant 1 +\alpha \genfrac{\lbrack}{\rbrack}{0pt}{}{D-1}{1}$, with equality if and only if $a_D=0$.
\end{lemma}

We now consider $1$-homogeneous distance-regular graphs with classical parameters and $a_1 > 0$, and obtain the following result.

\begin{theorem}\label{classical}
Let $\Gamma$ be a $1$-homogeneous distance-regular graph with classical parameters $(D,b,\alpha, \beta)$ and $a_1>0$, where $D \geqslant 5$ and $b \geqslant 1$. Then one of the following holds:
\begin{enumerate}[\normalfont(i)]
  \item $\alpha=0$,
  \item $\Gamma$ is a Johnson graph $J(2D,D)$,
  \item $\Gamma$ is a halved $\ell$-cube with $\ell \in \{2D,2D+1\}$,
  \item $\Gamma$ is a folded Johnson graph $\bar{J}(4D,2D)$,
  \item $\Gamma$ is a folded halved $4D$-cube,
  \item $D \leqslant 9$, $\alpha>0$, $b\geqslant2$.
\end{enumerate}
\end{theorem}

\begin{proof}
Let $\{\theta_i\}^D_{i=0}$ be the eigenvalues of $\Gamma$ as in \eqref{eq:eig CP}. 
Since $b\geqslant 1$, it follows $\theta_0 > \theta_1> \cdots > \theta_D$.
We observe that $b = {b_1}/(\theta_1 + 1)$ by evaluating $b_i$ in \eqref{clasic parameters} and $\theta_i$ in \eqref{eq:eig CP} at $i=1$.
Moreover, $c_2 \geqslant 2$ since every distance-regular graph with classical parameters and diameter at least three satisfies $c_i < c_{i+1}$ for $0\leqslant i \leqslant D-1$ (cf. \cite[Theorem 6.1.2]{BCN}).
Therefore, by Theorem \ref{thm:main}, $\Gamma$ is either a regular near $2D$-gon or falls into one of (ii)--(vi) therein.

\smallskip
We consider the parameter $\alpha$. 
Evaluating $c_i$ in \eqref{clasic parameters} at $i=2$ gives $c_2=(1+b)(1+\alpha)$, which implies that $\alpha(1+b)$ is an integer. 
Next, evaluating $c_i$ in \eqref{clasic parameters} at $i=3$ gives $c_3 = (1+b+b^2)(1+\alpha(1+b))$. Since $1+b+b^2 > 0$ and $c_3 > 0$, it follows $1+\alpha(1+b) > 0$.
As $\alpha(1+b)$ is an integer, we have $\alpha(1+b) \geqslant 0$.
Since $1+b > 0$, we obtain $\alpha \geqslant 0$.

\smallskip
We consider the case where $\alpha>0$; otherwise, we obtain (i).
Since $b\geqslant 1$, we divide the argument into two cases: either $b=1$ or $b\geqslant2$. 
If $b=1$, all distance-regular graphs with classical parameters $(D,1,\alpha,\beta)$ are known: the Gosset graph (with $\alpha =4$), the Johnson graphs (with $\alpha =1$) and the halved $\ell$-cubes, where $\ell \in \{2D,2D+1\}$ (with $\alpha =2$), the folded Johnson graphs (with $\alpha=1$), and the folded halved $4D$-cubes (with $\alpha=2)$; cf {\cite[Sections 6.1, 6.3]{BCN}}. 
However, we rule out the Gosset graph since it has diameter $3 < 5$. 
Therefore, by Theorem \ref{thm:main}, $\Gamma$ is a Johnson graph $J(2D,D)$, a halved $\ell$-cube with $\ell \in \{2D,2D+1\}$, a folded Johnson graph $\bar{J}(4D,2D)$, or
a folded halved $4D$-cube.

\smallskip
Next, consider the case where $b\geqslant2$ with $\alpha>0$. 
Since $\alpha(1 + b)$ is a positive integer, it follows that $\alpha \geqslant 1/(1 + b)$.
Using this together with Lemma \ref{beta}, we have $\beta > \alpha \genfrac{\lbrack}{\rbrack}{0pt}{}{D-1}{1} \geqslant \frac{1}{1+b}\genfrac{\lbrack}{\rbrack}{0pt}{}{D-1}{1}$.
Note that $k=\beta \genfrac{\lbrack}{\rbrack}{0pt}{}{D}{1}$.
By these comments, we have
\begin{align*}
	k = \beta \genfrac{\lbrack}{\rbrack}{0pt}{}{D}{1}
	> \frac{1}{1+b}\genfrac{\lbrack}{\rbrack}{0pt}{}{D}{1}\genfrac{\lbrack}{\rbrack}{0pt}{}{D-1}{1}
	= \frac{(b^D-1)(b^{D-1}-1)}{(b+1)(b-1)^2}.
\end{align*}
Using a computer, we find that $\frac{(b^D-1)(b^{D-1}-1)}{(b+1)(b-1)^2} >F(b)$ for $D\geqslant 10$, where $F(b)$ is from \eqref{poly: F}\footnote{
We note that the inequality $\frac{(b^D-1)(b^{D-1}-1)}{(b+1)(b-1)^2} >F(b)$ also holds for $D=8$ and $b\geqslant 6$, and for $D=9$ and $b\geqslant 3$.}.
In other words, for $D\geqslant 10$ we have $k > F(b)$.
However, by Theorem \ref{thm:main}, the inequality $k \leqslant F(b)$ must hold. 
Therefore, we have $D \leqslant 9$. 
The proof is complete.
\end{proof}

\begin{remark}
(i) We give a comment on the case (i) of Theorem \ref{classical}. 
In a subsequent paper, we will show that if $\Gamma$ is a $1$-homogeneous distance-regular graph with $a_1>0$ and classical parameters $(D,b,0,\beta)$ where $D\geqslant 5$, then $\Gamma$ is locally a disjoint union of cliques. 
Therefore, by Lemma \ref{locally disconnected}, $\Gamma$ is a regular near $2D$-gon. 
Consequently, according to {\cite[Theorem 9.4.4]{BCN}}, $\Gamma$ is either a dual polar graph or a Hamming graph. 

\smallskip
(ii) The unitary dual polar graphs $U(2D,r)$\footnote{This family of dual polar graphs is also denoted by ${}^2A_{2D-1}(r)$.} have two distinct types of classical parameters:
\begin{align}
	& (D, b, \alpha, \beta) = \left( D, q^2, 0, q \right), \label{CP_DPG:(1)}\\
	& (D, b, \alpha, \beta) = \left( D, -q, \frac{q(1+q)}{1-q}, \frac{q\left(1+(-q)^D\right)}{1-q}\right), \label{CP_DPG:(2)}
\end{align}
where $q$ is a prime power and $r = q^2$; see \cite[Section 6.2]{BCN}.
The graphs $U(2D, r)$ are $1$-homogeneous with $a_1 = q - 1 > 0$, and as such, when $D \geqslant 5$, they correspond to instances of case (i) of Theorem \ref{classical} under the parametrization \eqref{CP_DPG:(1)}.
The graph $U(2D, r)$ with classical parameters \eqref{CP_DPG:(2)} corresponds to none of the cases of Theorem 6.2 because $\alpha \neq 0$ and $b < 0$.
\end{remark}

%%%%%%%%%%%%%%%%%%%%%%%%%%%%%%%%%%%%%%%%%%%%%%%%%%%
%%%%%%%%%%%%%%%%%%%%%%%%%%%%%%%%%%%%%%%%%%%%%%%%%%%
%%%%%%%%%%%%%%%%%%%%%%%%%%%%%%%%%%%%%%%%%%%%%%%%%%%
%%%%%%%%%%%%%%%%%%%%%%%%%%%%%%%%%%%%%%%%%%%%%%%%%%%
\section{Tight distance-regular graphs}\label{sec:tightDRGs}

Let $\Gamma$ be a distance-regular graph with $D\geqslant 3$ and eigenvalues $k=\theta_0> \theta_1 > \cdots > \theta_D$.
Juri\v si\'c, Koolen and Terwilliger established the following so-called fundamental bound for $\Gamma$:
\begin{equation}\label{fundamental bound}
	\left( \theta_1 +\frac{k}{a_1 +1} \right)\left( \theta_D +\frac{k}{a_1 +1}\right ) \geqslant -\frac{ka_1b_1}{(a_1 +1)^2},
\end{equation}
cf. \cite[Theorem 6.2]{JKT2000}.
The graph $\Gamma$ is said to be \emph{tight} whenever $\Gamma$ is not bipartite and equality holds in \eqref{fundamental bound}.
We note that if $\Gamma$ is tight, then $a_1 > 0$; cf. \cite[Corollary 6.3]{JKT2000}.
We recall some characterizations of tight distance-regular graphs.
\begin{lemma}[cf. {\cite[Theorem 11.7, Theorem 12.6]{JKT2000}}]\label{char tightDRGs}
Let $\Gamma$ be a distance-regular graph with diameter $D\geqslant 3$. 
Then the following are equivalent.
\begin{enumerate}[\normalfont(i)]
	\setlength\itemsep{0pt}
	\item $\Gamma$ is tight.
	\item $\Gamma$ is $1$-homogeneous with $a_1 > 0$ and $a_D =0$.
	\item Every local graph of $\Gamma$ is connected strongly regular with eigenvalues $a_1>r>s$, where
	$$
		r= -1-\frac{b_1}{\theta_D+1}, \qquad s = -1-\frac{b_1}{\theta_1+1}.
	$$
\end{enumerate}
\end{lemma}

\medskip
In the following result, we characterize tight distance-regular graphs with diameter $D \geqslant 5$.

\begin{theorem}\label{tight corollary}
Let $\Gamma$ be a tight distance-regular graph with diameter $D \geqslant 5$. Let $b=b_1/(\theta_1 +1)$. Then $b \geqslant 1$ and one of the following holds:
\begin{enumerate}[\normalfont(i)]
  \item $\Gamma$ is a Johnson graph $J(2D,D)$,
  \item $\Gamma$ is a halved $2D$-cube,
  \item $\Gamma$ is locally connected with $k \leqslant F(b)$, where $F(b)$ is from \eqref{poly: F}.
\end{enumerate}
\end{theorem}
\begin{proof}
By Lemma \ref{char tightDRGs}, the tight graph $\Gamma$ is locally connected, which implies $c_2\geqslant 2$.
Moreover, $\Gamma$ is $1$-homogeneous with $a_1 > 0$ and $a_D=0$.
Therefore, by Theorem \ref{thm:main}, it follows that $b\geqslant 1$, and $\Gamma$ is either a regular near $2D$-gon or falls into one of cases (ii)--(vi) in Theorem \ref{thm:main}.
However, $\Gamma$ cannot be a regular near $2D$-gon since $a_1 > 0$ and $a_D=0$.
Thus, $\Gamma$ belongs to one of cases (ii)--(vi) therein.

\smallskip
Assume that $k > F(b)$; otherwise, it leads to case (iii).
By Lemma \ref{char tightDRGs}, the tight graph $\Gamma$ is locally a connected strongly regular graph.
We note that $\Gamma$ has $a_D=0$.
In a similar manner to the proof of Theorem \ref{thm:main}, we find that $\Gamma$ is locally either a Latin square graph or a Steiner graph.

\smallskip
If $\Gamma$ is locally a Latin square graph, by Theorem \ref{thm:main}, $\Gamma$ is either a Johnson graph $J(2D,D)$ or a folded Johnson graph $\bar{J}(4D,2D)$. 
However, we rule out the case of the graph $\bar{J}(4D,2D)$ since it has $a_D\ne 0$. 
Therefore, $\Gamma$ is a Johnson graph $J(2D,D)$.

\smallskip
If $\Gamma$ is locally a Steiner graph, then by Theorem \ref{thm:main} $\Gamma$ is either a halved $\ell$-cube with $\ell \in \{2D,2D+1\}$ or a folded halved $4D$-cube. 
The folded halved $4D$-cube cannot occur since it satisfies $a_D \ne 0$, and a tight distance-regular graph must have $a_D = 0$. 
Similarly, the case $\ell = 2D+1$ is excluded because the halved $(2D+1)$-cube also satisfies $a_D \ne 0$. 
Therefore, the only remaining possibility is that $\Gamma$ is a halved $2D$-cube.
This completes the proof.
\end{proof}

\begin{remark}
(i) In cases (i) and (ii) of Theorem \ref{tight corollary}, we have $b = 1$.
Also, note that, except for the halved $2D$-cubes and the Johnson graphs $J(2D, D)$, all known tight distance-regular graphs have diameter $D \leqslant 4$.

\medskip
(ii) Suppose $\Gamma$ is tight. 
In \cite[Theorem 1.3]{KLLLLT2023+}, Koolen et al. showed that if a local graph of $\Gamma$ is neither the block graph of an orthogonal array nor the block graph of a Steiner system, then the valency $k$ (and hence diameter $D$) of $\Gamma$ is bounded by a function in $b$, where $b=b_1/(1+\theta_1) \geqslant 2$. 
They then proposed a conjecture that generalizes this result: if $\Gamma$ is a tight distance-regular graph with $b \geqslant 2$, then the diameter $D$ of $\Gamma$ is bounded by a function in $b$; see \cite[Conjecture 28]{KLLLLT2023+}. 
Since the diameter of a distance-regular graph is bounded in terms of its valency (cf. \cite[Section 4]{BDKM}), it follows that Theorem \ref{tight corollary}(iii) proves this conjecture.
\end{remark}

We finish this section with the following corollary.
\begin{corollary}
Let $\Gamma$ be a tight distance-regular graph of diameter $D \geqslant 5$ with classical parameters $(D, b, \alpha, \beta)$. 
Then $b \geqslant 1$ and one of the following holds:
\begin{enumerate}[\normalfont(i)]
  \item $\Gamma$ is a Johnson graph $J(2D,D)$,
  \item $\Gamma$ is a halved $2D$-cube,
  \item $D \leqslant 9$, $\alpha>0$, and $b\geqslant 2$.
\end{enumerate}
\end{corollary}
\begin{proof}
By Theorems \ref{classical} and \ref{tight corollary}.
\end{proof}

%%%%%%%%%%%%%%%%%%%%%%%
%%%%%%%%%%%%%%%%%%%%%%%
\section*{Acknowledgements}
The authors thank the anonymous referees for the valuable comments.
J.H. Koolen is partially supported by the National Key R. and D. Program of China (No. 2020YFA0713100), the National Natural Science Foundation of China (No. 12071454), and the Anhui Initiative in Quantum Information Technologies (No. AHY150000). 
M. Abdullah is supported by the Chinese Scholarship Council at USTC, China.
B. Gebremichel is supported by the National Key R. and D. Program of China (No. 2020YFA0713100) and the Foreign Young Talents Program (No. QN2022200003L).
J.-H. Lee was partially supported by the Mathematics Department at Pohang University of Science and Technology (POSTECH). J.-H. Lee completed part of this work during his sabbatical visit to POSTECH and expresses deep gratitude for their hospitality.

\end{document}